\documentclass[12pt, twoside, leqno]{article}
\usepackage{}

\usepackage{amsmath,amsthm}
\usepackage{amssymb}
\usepackage{hyperref}

\usepackage{enumerate}

\usepackage{graphicx}

\newtheorem{thm}{Theorem}[section]
\newtheorem{cor}[thm]{Corollary}
\newtheorem{lem}[thm]{Lemma}
\newtheorem{prop}[thm]{Proposition}

\theoremstyle{definition}
\newtheorem{defin}[thm]{Definition}

\numberwithin{equation}{section}

\newcommand{\R}{\mathbb{R}}

\newcommand{\eps}{\varepsilon}

\begin{document}


\baselineskip=17pt


\title{Subdifferential representation of convex functions on $X^*$}

\author{Duanxu Dai\\
College of Mathematics and Computer Science\\
Quanzhou Normal University\\
Quanzhou 362000, China\\
E-mail: dduanxu@163.com
\and}

\date{}

\maketitle


\renewcommand{\thefootnote}{}



\footnote{ Supported by the Natural Science Foundation of China (Grant No. 11601264) and the Outstanding Youth Scientific Research Personnel Training Program of Fujian Province and the Research Foundation of Quanzhou Normal University(Grant No. 2016YYKJ12) and the High level Talents Innovation and Entrepreneurship Project of Quanzhou City.}
\footnote{}

\renewcommand{\thefootnote}{\arabic{footnote}}
\setcounter{footnote}{0}


\begin{abstract}
 In this paper,  we obtain subdifferential representation of a proper $w^*$-lower semicontinous convex function on $X^*$ as follows:
Let $g$ be a proper convex $w^*$-lower semicontinuous function
on $X^*$. Assume that int dom $g$ $\neq\emptyset$ (resp. int (dom ($g^*|_X)$)$\neq\emptyset$). Then given any point $x_0^*$ $\in$ D ($\partial g\cap X$)
and $x^*$ $\in$ dom $g$ (resp. $x^*\in X^*$), we have
$$g(x^*)=g(x_0^*)+\sup\{\sum_{i=0}^{n-1}\langle x_i,x_{i+1}^*-x_i^*\rangle +\langle x_n,x^*-x_n^*\rangle \},$$
where the above supremum is taken over all integers $n$, all $x_i^*\in X^*$ and all $x_i\in\partial g(x_i^*)\cap X$ for $i=0,1,\cdots,n$.
(resp. if, moreover, $X^*$ has the Radon-Nikodym property, then we may estimate the above supremum among the set of $w^*$-strongly exposed points of $g$.)
\end{abstract}

\section{Introduction}

Let $X$ be a Banach space and $X^*$ be the dual of $X$. Let $g:X\rightarrow R\cup\{+\infty\}$ be a proper lower semicontinuous convex
function. Rockafellar \cite{Roc} has shown that $g$ can be represented through its subdifferential $\partial g$ as follows:
\begin{align}\label{E1}
g(x)&= g(x_0)+\sup\{\sum_{i=0}^{n-1}\langle x_i^*,x_{i+1}-x_i\rangle +\langle x_n^*,x-x_n\rangle \}
\end{align}
for every $x\in X$, where the above supremum is taken over all integers $n$, all $x_i\in X$ and all $x_i^*\in\partial g(x_i)$ for $i=0,1,\cdots,n$.

J. Benoist and A. Daniilidis \cite{Ben} show that, in Banach spaces with the Radon-Nikodym property, and
under the assumption that int (dom ($g^*$))$\neq\emptyset$ , the above formula (\ref{E1}) can be simplified.
Indeed, it suffices to estimate the above supremum among the set of strongly exposed
points of $g$, instead of the domain of $\partial g$.

In this paper, we consider the case where $g$ is a proper $w^*$-lower semicontinous convex function defined on a dual Banach space $X^*$. Thus, we show that if int dom $g$ $\neq$ $\emptyset$ (resp. int (dom $(g^*|_X)$) $\neq$ $\emptyset$), then given any point $x_0^*$ $\in$ D ($\partial g\cap X$)
and $x^*$ $\in$ dom $g$ (resp. $x^*\in X^*$), we have
\begin{align}\label{E2}
g(x^*)&=g(x_0^*)+\sup\{\sum_{i=0}^{n-1}\langle x_i,x_{i+1}^*-x_i^*\rangle +\langle x_n,x^*-x_n^*\rangle \},
\end{align}
where the above supremum is taken over all integers $n$, all $x_i^*\in X^*$ and all $x_i\in\partial g(x_i^*)\cap X$.
(resp. if, moreover, $X^*$ has the Radon-Nikodym property, then we may estimate the above supremum among the set of all $w^*$-strongly exposed points of $g$, instead of the much larger set of all strongly exposed points of $g$.)

\section{Preliminaries}

We list the following classical definitions (See, for
example \cite{Phe}).

\begin{defin}Let $f:X\rightarrow R\cup\{\infty\}$. The
effective domain of $f$ is the set dom ($f$) $=$ $\{x\in X:
f(x)<+\infty\}$. We say that $f$ is lower semicontinuous (l.s.c.)
provided $\{x\in X: f(x)\leq r\}$ is closed in $X$ for every $r\in
R$, i.e., the epigraph of $f:$
 $$\{(x,r)\in X\times R: r\geq f(x)\}$$
 is closed in $X\times R$. Equivalently, $f$ is l.s.c. provided
  $$f(x)\leq \liminf f(x_\delta)$$
  whenever $x\in X$ and $(x_\delta)$ is a net in $X$ converging to
  $x$.

  We say $f$ is proper if both $f(x)\neq -\infty$ for all $x\in X$ and dom ($f$) $\neq$ $\emptyset$.
\end{defin}

\begin{defin}Let $f$ be a proper function from $X$ into $(-\infty, +\infty]$.
The Fenchel conjugate $f^*$ of $f$ is defined by
$$f^*(x)=\sup\{\langle x^*, x\rangle-f(x):x\in X\},\;\rm{for\;all}\;x^*\in X^*.$$
then $f^*$ is a $w*-l.s.c$ function from $X^*$ into $(-\infty,
+\infty]$ and may be not proper.
\end{defin}

\begin{defin}Let $f:X\rightarrow(-\infty, +\infty]$ be a proper convex and
l.s.c-function. If $x$ $\in$ dom ($f$), define the subdifferential
$\partial f(x)$ by

$$\partial f(x)=\{x^*\in X^*:\langle x^*,y-x\rangle\leq f(y)-f(x)\;\rm{for\;all}\;y\in X\}$$
$$\;\;\;\;\;\;\;\;\;\;\;\;=\{x^*\in X^*:\langle x^*,y\rangle\leq f(x+y)-f(x)\;\rm{for\;all}\;y\in X\}.\;\;$$
and $x^*\in\partial f(x)$ if and only if $x^*(y)\leq d^{+}f(x)(y)$
for all $y\in X$. This set may be empty.
\end{defin}

\begin{defin}Let $f:X\rightarrow(-\infty, +\infty]$ be a proper convex and l.s.c.
function. Suppose $x$ $\in$ dom ($f$). For any $\eps>0$, define the
$\eps$-subdifferential $\partial_{\eps} f(x)$ by

$$\partial_{\eps}f(x)=\{x^*\in X^*:\langle x^*,y-x\rangle\leq f(y)-f(x)+\eps\;\rm{for\;all}\;y\in X\}$$
and $x^*\in\partial_\eps f(x)$ if and only if $x^*(y)\leq
d^{+}f(x)(y)+\eps$ for all $y\in X$. This set is always nonempty and
$w^*$-closed.
\end{defin}

We list in the following statement some easy facts\cite{Fab} about
conjugate functions. The proof is standard.

\begin{prop}\label{P1}Let $f$ and $g$ be two proper functions from $X$ into $(-\infty, +\infty]$.
Then, given $x\in X$ and $x^*\in X^*$,

$i)\;f(x)+f^*(x^*)\geq x^*(x).$

$ii)\;f^*\; is\; proper,\; then \;f^{**}|_X\leq f.$

$iii)\;f\leq g,\; then \;f^{*}\geq g^*.$

\end{prop}

\begin{prop}\label{P2}Let $f: X\rightarrow(-\infty, +\infty]$ be a proper function
such that $f^*$ is also proper. Then

i) $\rm{epi}$ $f^{**}$ $=$ $\overline{\rm{conv}} ^{w^*}$($\rm{epi}$ $f$).

ii) $f^{**}|_X$ $=$ $f$ if and only if $f$ is convex and l.s.c.
\end{prop}

\begin{cor}\label{Cor1} Every convex and $w^*$-l.s.c function $g: X^*\rightarrow(-\infty,
+\infty]$ is the Fenchel conjugate of $g^*|_X$.
\end{cor}

\begin{prop}\label{P3}Let $f: X\rightarrow(-\infty, +\infty]$ be a proper
function. Let $x\in X$, $x^*\in X^*$ and $\eps>0$.Then

i) $x^*\in \partial_{\eps}
f(x)$ if and only if $f(x)+f^*(x^*)\leq x^*(x)+\eps$.

ii) $x^*\in \partial f(x)$ if and only if $f(x)+f^*(x)= x^*(x)$.
\end{prop}

\begin{cor}\label{Cor2} Let $f: X\rightarrow(-\infty, +\infty]$ be a proper
function. Let $x\in X$, $x^*\in X^*$ and $\eps\geq0$.Then

i) if $x^*\in \partial_{\eps} f(x)$, then $x\in\partial_{\eps}
f^*(x^*)$.

ii) if $f$ is moverover, convex and l.s.c., and $x\in\partial_{\eps}
f^*(x^*)$, then $x^*\in \partial_{\eps} f(x)$.

\end{cor}

J. Borwein \cite{Borw} prove a theorem which provides some formulae and extends the Bishop-Phelps-Bronsted-Rockafellar theorem.

In this section, following the techniques of J. Borwein \cite{Borw} we will give a dual version of J. Borwein's theorem which will be applied to subdifferential representation of a proper $w^*$-lower semicontinous convex function on $X^*$.

Through the similar reasoning, we have the following dual version of Proposition 3.15 in \cite{Phe}.
\begin{prop}\label{P6}If $g$ is a proper convex $w^*$-lower semicontinuous function
on $X^*$ then $\partial_{\eps} g(x^*)\cap X\neq\emptyset$, for every
$x^*$ $\in$ $\rm{dom}$ ($g$) and every $\eps>0$.
\end{prop}

Now, we have the following dual version of J.Borwein's theorem.
\begin{thm}\label{T1} Assume that $g$ is a proper convex $w^*$-lower semicontinuous function
on $X^*$, and suppose that $\varepsilon>0$, $\beta\geq0$ and
$x_0^*$ $\in$ $\rm{dom}$ ($g$). Suppose, further, that $x_0\in\partial_\eps
g(x_0^*)\cap X$. Then there exist points $x_\eps^*$ $\in$ $\rm{dom}$ ($g$) and
$x_\eps\in X$ such that

$i)\;x_{\eps}\in \partial g(x_{\eps}^*)\cap
X;\;\;\;\;ii)\;\|x_{\eps}^*-x_0^*\|
\leq\sqrt{\eps}(1+\beta\|x_0^*\|);$

$iii)\;|g(x_{\eps}^*)-g(x_0^*)| \leq\sqrt{\eps}(\|x_0\|+\beta
|\langle x_0^*,x_0\rangle|)+2\eps;$

$iv)\;\|x_{\eps}-x_0\| \leq\sqrt{\eps};$

$v)\;|\langle x_{\eps},x^*\rangle -\langle
x_0,x^*\rangle|\leq\sqrt{\eps}\|x^*\|\;for\;all\;x^*\in X^*\;and$

$vi)\;x_0\in\partial_{2\eps} g(x_{\eps}^*).$

$vii)\;|\langle x_{\eps}-x_0,x_{\eps}^*-x_0^*\rangle|\leq\eps$

\end{thm}

\begin{proof}Assume that $g$ is a proper convex $w^*$-lower semicontinuous function
on $X^*$, then by Corollary \ref{Cor1}, $g=(g^*|_X)^*$. We write
$f=g^*|_X$, then $g=f^*$ and that $f$ is proper, convex and lower
semicontinuous function on $X$.

Assume that $x_0\in\partial_{\eps}g(x_0^*)\cap
X$=$\partial_{\eps}f^*(x_0^*)\cap X$, then by Corollary \ref{Cor2},
$x_0^*\in\partial_{\eps}f(x_0)$. Apply J.Borwein's theorem to $f$
and Banach space $X$. Then there exist points $x_{\eps}$ $\in$ dom ($f$)
and $x_{\eps}^*\in X^*$ such that

$1)\;x_{\eps}^*\in \partial
f(x_{\eps});\;\;\;\;\;\;\;\;\;\;2)\;\|x_{\eps}-x_0\|
\leq\sqrt{\eps};$

$3)\;|f(x_{\eps})-f(x_0)| \leq\sqrt{\eps}(\sqrt{\eps}+1/\beta);$

$4)\;\|x_{\eps}^*-x_0^*\| \leq\sqrt{\eps}(1+\beta\|x_0^*\|) ;$

$5)\;|\langle x_{\eps}^*,x\rangle -\langle
x_0^*,x\rangle|\leq\sqrt{\eps}(\|x\|+\beta |\langle
x_0^*,x\rangle|)\;\rm{for\;all}\;$$x\in X$

$6)\;x_{\eps}^*\in\partial_{2\eps} f(x_0)\;\rm{and}$

$7)\;|\langle x_{\eps}^*-x_0^*,x_{\eps}-x_0\rangle|\leq\eps.$

Hence  i), ii), iv), v), vi)and vii) is obvious and that it suffices to
prove iii).

By Proposition \ref{P1} and Proposition \ref{P3}, since
$x_0^*\in\partial_{\eps}f(x_0)$ and $x_{\eps}^*\in\partial
f(x_{\eps})$ we have

$(a)\;x_0^*(x_0)\leq f(x_0)+f^*(x_0^*)\leq x_0^*(x_0)+\eps;$

$(b)\;f(x_{\eps})+f^*(x_{\eps}^*)= x_{\eps}^*(x_{\eps});$

$(c)\;x_{\eps}^*(x_0-x_{\eps})\leq f(x_0)-f(x_{\eps})\leq
x_0^*(x_0-x_{\eps})+\eps$.

Hence,
$$x_0^*(x_0)-f(x_0)\leq f^*(x_0^*)\leq
x_0^*(x_0)-f(x_0)+\eps$$
and
$$g(x_{\eps}^*)-g(x_0^*)=f^*(x_{\eps}^*)-f^*(x_0^*)=x_{\eps}^*(x_{\eps})-f(x_{\eps})-f^*(x_0^*).$$
So
$$g(x_{\eps}^*)-g(x_0^*)\leq x_{\eps}^*(x_{\eps})-f(x_{\eps})-x_0^*(x_0)+f(x_0)$$

$$\;\;\;\;\;\;\;\;\;\;\;\;\;\;\;\;\;\;\;\;\;\;\leq x_{\eps}^*(x_{\eps})-x_0^*(x_0)+x_0^*(x_0-x_{\eps})+\eps$$

$$=\langle x_{\eps}^*-x_0^*, x_{\eps}\rangle +\eps\;$$
and
$$g(x_{\eps}^*)-g(x_0^*)\geq f(x_0)-x_0^*(x_0)-\eps+x_{\eps}^*(x_{\eps})-f(x_{\eps})$$

$$\;\;\;\;\;\;\;\;\;\;\;\;\;\;\;\;\geq x_{\eps}^*(x_{\eps})-x_0^*(x_0)+x_{\eps}^*(x_0-x_{\eps})-\eps$$

$$=\langle x_{\eps}^*-x_0^*,x_0\rangle-\eps.\;\;\;\;\;\;\;$$
Since 7)
$$|\langle x_{\eps}^*-x_0^*,x_{\eps}-x_0\rangle|\leq\eps$$
and 5)
$$|\langle
x_{\eps}^*-x_0^*,x_0\rangle|\leq\sqrt{\eps}(\|x_0\|+\beta |\langle
x_0^*,x_0\rangle|),$$
we have
$$|\langle x_{\eps}^*-x_0^*,x_{\eps}\rangle|\leq\sqrt{\eps}(\|x_0\|+\beta
|\langle x_0^*,x_0\rangle|)+\eps.$$
Hence
$$\;|g(x_{\eps}^*)-g(x_0^*)| \leq\sqrt{\eps}(\|x_0\|+\beta
|\langle x_0^*,x_0\rangle|)+2\eps$$
and we complete the proof.
\end{proof}

As a corollary, we obtain the following dual version of Theorem 3.17 in\cite{Phe} which was raised by Fabian during his visit to Xiamen University in 2012.
\begin{thm}\label{T0} Let $g$ be a proper convex $w^*$-lower semicontinuous function
on $X^*$. Then given any point $x_0^*$ $\in$ $\rm{dom}$ ($g$), $\varepsilon>0$
and $x_0\in\partial_\eps g(x_0^*)\cap X$ there exist points
$x_\eps^*$ $\in$ $\rm{dom}$ ($g$) and $x_\eps\in X$ such that

$$\;\;x_{\eps}\in \partial g(x_{\eps}^*)\cap
X,\;\;\|x_{\eps}^*-x_0^*\|\leq\sqrt{\eps}\;\;\;and
\;\;\;\|x_{\eps}-x_0\|\leq\sqrt{\eps}.\;$$
\end{thm}

\begin{cor}Let $U$ be an open convex subset of a dual Banach
space $X^*$ and let $g:U \rightarrow \mathbb{R}$ be a convex
$w^*-l.s.c$-function. Then the set $\{x^*\in U:\partial g(x^*)\cap
X\neq \emptyset\}$ is dense in $U$. Moreover, if $g$ is Frechet
differentiable at some $x^*\in U$, then the derivative $g'(x^*)$
belongs to $X$.
\end{cor}

\begin{proof} For the first statement, fix any point $x_0^*\in U$ and any $\eps>0$ so small that $B(x_0^*,\eps)\subset U$. Define $f: X^*\rightarrow
(-\infty,+\infty]$ by $f(x^*)=g(x_0^*+x^*)$ if $x^*\in \eps B_{X^*}$
and $f(x^*)=+\infty$ if $x^* \notin \eps B_{X^*}$. Then $f$ is
proper, convex and $w^*-l.s.c$ on $X^*$.

since $0$ $\in$ dom ($f$), by Proposition \ref{P6} there exists
$x_0\in\partial_{\eps^2} f(0)\cap X$, then by Theorem \ref{T0} there
exist point $x_{\eps^2}^*$ $\in$ dom ($f$) and $x_{\eps^2}\in X$ such that

$$\;\;x_{\eps^2}\in
\partial f(x_{\eps^2}^*), \;\; \|x_{\eps^2}^*-0\|\leq \eps\;\;and\;\;\|x_{\eps^2}-x_0\|\leq\eps.$$

Hence

$x_0^*+x_{\eps^2}^*\in B(x_0^*,\eps)\subset U$ and $x_{\eps^2}\in
\partial g(x_0^*+x_{\eps^2}^*)$

For the second statement. Assume that $g$ is Frechet differentiable
at some $x_0^*\in U$, then by the first statement and Proposition \ref{P6}, for every $\eps>0$, there exists $x_{\eps}\in X$ such that
$$x_{\eps}\in
\partial_{\eps} f(0)=\partial_{\eps} g(x_0^*)$$

and let $\eps\rightarrow 0$, then by $\check{S}$mulyan lemma,
$x_{\eps}\rightarrow g'(x_0^*)$ so $g'(x_0^*)\in X$.
\end{proof}

\section{The Main Results}

In this section, we introduce some new notions such as $w^*$-monotone, $w^*$-maximal monotone, $w^*$-cyclically monotone and $w^*$-maximal cyclically monotone, and use some techniques from J. Benoist and A. Daniilidis \cite{Ben} to obtain the subdifferential representation of $w^*$-lower semicontinuous convex functions on $X^*$.

\subsection{$w^*$-Strongly exposed points and Radon-Nikodym property}

The definitions of ($w^*$-)strongly exposed points and ($w^*$-)strongly exposed functionals of a convex set on $X$($X^*$) can be found in \cite{Ben} and \cite{Phe}.

\begin{defin}\label{D1}Let $C$ be a non-empty closed convex subset of $X^*$. A point $u^* \in C$ is
said $w^*$-strongly exposed if there exists $u\in X$ such that for each sequence $\{u_n^*\}\subset C$ the
following implication holds
$$\lim_{n\rightarrow +\infty }\langle u_n^*,u\rangle=\sigma_C(u)\;\Rightarrow \;\lim_{n\rightarrow +\infty }u_n^*=u^*$$
\end{defin}
In such a case $u$ is said to be a $w^*$ strongly exposing functional for the point $u^*$ in $C$.
We denote by $w^*$-Exp $(C, u^*)$ the set of all functionals of X satisfying this property.
Let us further denote by $w^*$-exp $C$ the set of $w^*$ strongly exposed points of $C$. We also denote by $w^*$-Exp $C$ the set of all $w^*$ strongly exposing functionals.

\begin{defin} A point $x^*$ $\in $ dom $g$ is called $w^*$-strongly exposed for the proper lower semi-
continuous convex function g if
$$((x^*,g(x^*))\in w^*-\rm{exp\;(epi}\; g)$$
\end{defin}
We denote by $w^*$-exp $g$ the set of $w^*$-strongly exposed points of $g$.
For every $x^*$ $\in$ $w^*$-exp $g$ we denote by $w^*$-Exp $(g, x^*)$ the set of all $x\in X$ satisfying
$$(x, -1)\in w^*-\rm{Exp\;(epi \;}g, (x^*,g(x^*))). $$
Obviously, $w^*$-Exp $(g, x^*)$ $\subset$ $\partial g(x^*)\cap X$. We also denote by $w^*$-Exp $g$ the set of $w^*$-strongly exposed functionals of $g$. Indeed,
$$w^*-\rm {Exp}\; g=\bigcup_{x^*\in w^*-\rm{exp}\; g} w^*-\rm{Exp}\; (g, x^*).$$

\begin{defin}A Banach space $X$(resp.$X^*$) is said to have the Radon-Nikodym property, if every
non-empty bounded closed (resp. $w^*$-compact) convex subset $C$ of $X$ (resp. $X^*$) can be represented as the ($w^*$)-closed convex
hull of its (resp. $w^*$)-strongly exposed points, that is,
$$C = \overline{\rm{co}}(\rm{exp}\;C)\;\;(resp. \;C = \overline{\rm{co}}^{w^*}(w^*-\rm{exp}\;C)).$$
\end{defin}

\begin{defin}A proper $w^*$-lower semicontinuous function $g : X^* \rightarrow \R\cup\{+\infty\}$ is called
$w^*$-epi-pointed if
int (dom $g^*|_X$) $\neq$ $\emptyset$.
\end{defin}

Through the similar deduction of [\cite{Ben}, Proposition 2.3, Proposition 3.6], we have the following two propositions, respectively.
\begin{prop}\label{P4}Suppose that $X^*$ has the Radon-Nikodym property and $C$ is a nonempty
closed convex set. Then $w^*$-$\rm{Exp}$ $C$ is dense in $\rm{int\; dom}$ $\sigma_C|_X$.
\end{prop}

\begin{prop}\label{P5}The set $w^*$-$\rm{Exp}$ $g$ is dense in $\rm{int\;(dom}$ $g^*|_X)$ if the Banach space $X^*$ has the
Radon-Nikodym property and the convex function $g$ is $w^*$-epi-pointed.
\end{prop}

\subsection{$w^*$-Cyclically monotone operators}

The classical concepts of monotone, maximal monotone, cyclically monotone and maximal cyclically monotone can be found in \cite{Ben} and \cite{Phe}. In this section, we will introduce some new notions such as $w^*$-monotone, $w^*$-maximal monotone, $w^*$-cyclically monotone and $w^*$-maximal cyclically monotone. In section 3, we will show that $\partial
g\cap X$ is the unique $w^*$-maximal cyclically monotone operator, where $g$ is a proper convex $w^*$-lower semicontinuous function on $X^*$.

\begin{defin} A set-valued map $\Phi: X^*\rightarrow 2^X$ is
said to be a $w^*$-monotone operator provided

$$\langle x-y, x^*-y^*\rangle\geq0$$
whenever $x^*,y^*\in X^*$ and $x\in\Phi(x^*)$ and $y\in\Phi(y^*)$.
The effective domain $\rm{D}(\Phi)$ of $\Phi$ is defined by

$$\rm{D}(\Phi)=\{x^*\in X^*:\Phi(x^*)\neq\emptyset\}.$$

\end{defin}

\begin{defin} A set-valued map $\Phi: X^*\rightarrow 2^X$ is
said to be $w^*$ n-cyclically monotone operator provided

$$\sum_{1\leq k\leq n}\langle x_k, x^*_k-x^*_{k+1}\rangle\geq0$$
whenever $n\geq 2$ and $x^*_1$, $x^*_2$,$\cdots$,$x^*_n\in
\rm{D}(\Phi)$, $x^*_{n+1}=x^*_1$ and $x_k\in\Phi(x^*_k)$, $k=1,2,3,\ldots,n$.

We say that $\Phi$ is $w^*$-cyclically monotone if it is $w^*$
n-cyclically monotone for every n. Clearly, a $w^*$ 2-cyclically
monotone operator is $w^*$-monotone.
\end{defin}

\begin{defin}
A $w^*$-monotone operator $\Phi$ is said to $w^*$-maximal cyclically
monotone provided $\Phi=\Psi$ whenever $\Psi$ is $w^*$-cyclically
monotone and Gr ($\Phi$) $\subset$ Gr ($\Psi$). Clearly, a $w^*$-maximal
monotone operator which is $w^*$-cyclically monotone is necessarily
$w^*$-maximal cyclically monotone.
\end{defin}

\subsection{ Subdifferential representation of convex functions on $X^*$}
By the similar reasoning, we have the following two lemmas that may be considered as the dual versions of Lemma 3.22 and Lemma 3.23 in\cite{Phe}.

\begin{lem}\label{L1} Let $g$ be a proper convex $w^*-l.s.c$-function on $X^*$. If $\alpha,\beta>0$, $x_0^*\in X^*$ and $g(x_0^*)<\inf _{X^*}g+\alpha\beta$, then there exist $x^*\in X^*$ and $x\in \partial g(x^*)\cap X$ such that $\|x^*-x_0^*\|<\beta$ and $\|x\|<\alpha$.

\end{lem}

The following Lemma \ref{L2} is a consequence of Lemma \ref{L1}.
\begin{lem} \label{L2} Let $g$ be a proper convex $w^*-l.s.c$-function on $X^*$. If $x^*\in X^*$ such that $\inf _{X^*}g<g(x^*)$, then there exist $z^*$ $\in$ $\rm{dom}$ $g$ and $z\in\partial g(z^*)\cap X$ such that
$$g(z^*)<g(x^*) \;\;\rm {and}\;\;\langle z,x^*-z^*\rangle>0.$$

\end{lem}

\begin{prop}\label{P7} If $g$ is a proper
convex $w^*-l.s.c$-function on $X^*$, then $\partial g\cap X$ is  $w^*$-maximal monotone and $w^*$-maximal cyclically
monotone.
\end{prop}

\begin{proof}Since $g$ is a proper convex $w^*-l.s.c$-function on $X^*$, $\partial g$ is maximal monotone and maximal cyclically
monotone. Hence $\partial g\cap X$ is $w^*$-monotone and $w^*$-cyclically monotone. It suffices to show that $\partial g\cap X$ is $w^*$-maximal monotone. Suppose that $T: X^*\rightarrow 2^X$ is $w^*$-monotone satisfying that Gr $\partial g\cap X$ $\subset$ Gr $T$. Assume that there exists $(x^*,x)$ $\in$ Gr $T$ such that $(x^*,x)$ $\notin$ Gr $\partial g\cap X$. Therefore, it follows from $x\in \partial g(x^*)$ that $0\notin \partial (g-x)(x^*)$. Thus, $\inf_{X^*}(g-x)<(g-x)(x^*)$. By Lemma \ref{L2}, there exist $z^*$ $\in $ dom ($g-x$) $\equiv$ dom ($g$) and $z\in\partial(g-x)(z^*)\cap X$ such that $ \langle z,z^*-x^*\rangle<0$. Hence, there exists $y\in \partial g(z^*)\cap X$ such that $z=y-x$ and $\langle y-x,z^*-x^*\rangle<0$ which is a contradiction with the $w^*$-monotonicity of $T$. So we complete the proof.

\end{proof}

\begin{prop}\label{P8} Let $T: X^*\rightarrow 2^X$ be a $w^*$-cyclically
monotone operator and $x_0^*$ $\in$ $\rm{D}$ $(T)$. Let $h: X^*\rightarrow R\cup \{+\infty\}$ be defined for all $x^*\in X^*$ by
\begin{align}\label{E3} h(x^*)&=\sup\{\sum_{i=0}^{n-1}\langle x_i,x_{i+1}^*-x_i^*\rangle +\langle x_n,x^*-x_n^*\rangle \},\end{align}
where the above supremum is taken over all integers $n$, all $x_i^*\in X^*$ and all $x_i\in T(x_i^*)$. Then $h$ is a proper convex $w^*$-lower semicontinuous function such that \begin{align}\label{E4} \rm{Gr}\; T &\subset \rm{Gr}\;\partial h.\end{align}
\end{prop}

\begin{proof} Since $T: X^*\rightarrow 2^X$ is $w^*$ cyclically
monotone and $x_0^*$ $\in$ D $(T)$ and $x_0\in T(x_0^*)$, we have $h(x_0^*)<+\infty$ by letting $x^*=x_{n+1}^*=x_0^*$ and $x_{n+1}=x_0$.
Hence $h$ is proper. It is obvious that $h$ is convex and $w^*$-lower semicontinuous. Now we will prove that Gr $T$ $\subset$ Gr $\partial h$.
By Proposition \ref{P1} and Proposition \ref{P3}, it suffices to show for every $x^*\in X^*$ and $x\in T(x^*)$ that $h^*(x)+h(x^*)\leq x^*(x)$. By definition, for every $\lambda < h(x^*)$ there exist $x_i^*$ $\in$ D $T$ and $x_i\in T(x_i^*)$, $i=1,2,\cdots, n$ such that
$$\lambda<\sum_{i=0}^{n-1}\langle x_i,x_{i+1}^*-x_i^*\rangle +\langle x_n,x^*-x_n^*\rangle.$$
Let $x_{n+1}^*=x^*$ and $x_{n+1}=x $. For every $y^*\in X^*$,
$$h(y^*)\geq \sum_{i=0}^{n}\langle x_i,x_{i+1}^*-x_i^*\rangle +\langle x_{n+1},y^*-x_{n+1}^*\rangle$$
$$>\lambda+\langle x_{n+1},y^*-x_{n+1}^*\rangle$$
$$=\lambda+\langle x,y^*-x^*\rangle.$$
Hence
$$\langle x,y^*\rangle- h(y^*)\leq \langle x,x^*\rangle-h(x^*)$$
By the definition of $h^*$, we have
$$h^*(x)\leq \langle x,x^*\rangle-h(x^*)$$
So $x\in \partial h(x^*)$.
\end{proof}

By using the similar technique of Proposition \ref{P8} , we have the following Proposition \ref{P9}, combining with Proposition \ref{P7} which say $\partial
g\cap X$ is the unique $w^*$-maximal cyclically
monotone operator.
\begin{prop}\label{P9} If $\Phi: X^*\rightarrow 2^X$ is $w^*$-maximal cyclically
monotone, with $\rm{D}$ $(\Phi)$ $\neq$ $\emptyset$, then there exists a proper
convex $w^*-l.s.c$-function $g$ on $X^*$ such that $\Phi=\partial
g\cap X.$
\end{prop}

\begin{thm}\label{T2} Let $g$ be a proper convex $w^*$-lower semicontinuous function
on $X^*$ and let $T: X^*\rightarrow 2^X$ be a set valued operator such that \begin{align}\label{E5} \rm{Gr}\; T &\subset \rm{Gr}\; \partial g.\end{align}
Let $x_0^*$ $\in$ $\rm{D}$ $(T)$. Denote by $h$ the proper convex $w^*$-lower semicontinuous function defined above. Then

(i) If $\rm{int} $ $\rm{(dom}$ $g$$)$ $\neq$ $\emptyset$ and $\rm{D}$ $(T)$ is dense in $\rm{int\; (dom}$ $g$$)$, then for all $x^*$ $\in$ $\rm{dom}$ $g$,
 $$g(x^*)-g(x_0^*)=h(x^*).$$

(ii) If $\rm{int\; (dom}$ $g^*|_X$$)$ $\neq$ $\emptyset$ and $\rm{Im}$ $T$ is dense in $\rm{int\; (dom}$ $g^*|_X$$)$, then for all $x^*\in X^*$,
 $$g(x^*)-g(x_0^*)=h(x^*).$$
\end{thm}

\begin{proof}
(i) By (\ref{E1}), (\ref{E3}) and (\ref{E5}), it is clear that for all $x^*\in X^*$,
\begin{align}\label{E6} g(x^*)-g(x_0^*)&\geq h(x^*).\end{align}
Let $U$ $=$ int (dom $g$) $\neq$ $\emptyset$. By (\ref{E6}), we have $U$ $\subset$ dom $h$. Since $U$ is open, it follows from[3, Proposition 2.5] that $U$ $\subset$ int (D $(\partial g)$) $\cap$ int $($D $(\partial h)$$)$. Hence, by [\cite{Ben}, Proposition 2.4,Proposition 2.8], the maximal monotone operators $\partial g$ and $\partial h$ are minimal $w^*$-cuscos on $U$. By (\ref{E5}), we have
$$\rm{Gr}\; T\subset \rm{Gr}\; \partial g,$$
while by Proposition \ref{P8} we have
$$\rm{Gr}\; T\subset \rm{Gr}\; \partial h.$$
Since D $(T)$ is dense in $U$, Proposition 2.9 in \cite{Ben} yields that $\partial g=\partial h$ on $U$. Hence by \cite{Roc} there exists $r\in \R$ such that $g = h + r$ on $U$. By the $w^*$ lower semicontinuity of $g$ and $h$, the equality $g = h + r$ holds on $\rm {dom}\;g$. By the definition of $h$ and noting that the operator $T$ is $w^*$-cyclically monotone we have $h(x_0^*)=0$, hence we conclude that $g(x_0) = r$ and thus equality holds.

(ii) Let $V$ $=$ int (dom $g^*|_X$) $\neq$ $\emptyset$. By [\cite{Phe}, Theorem 2.28], $\partial (g^*|_X)$ is locally bounded on $V$. Obviously,
$$\rm{Gr}\;\partial g^{-1}\subset\rm{Gr}\;\partial (g^*|_X).$$
By (\ref{E5}), we have
$$\rm{Gr}\;T^{-1}\subset\rm{Gr}\; \partial (g^*|_X),$$
which yields that $T^{-1}$ is locally bounded on $V$. By [\cite{Ben}, Lemma 2.6], we have
$$V\subset\rm{int\; (dom}\;h^*|_X).$$
By Proposition \ref{P8}, we have $\rm{Gr}$ $T$ $\subset$ $\rm{Gr}$ $\partial h$, which implies that
$$\rm{Gr}\;T^{-1}\subset\rm{Gr}\;\partial (h^*|_X).$$
Since $\rm{D}$ $(T^{-1})$ $= $ $\rm{Im}$ $T$ is dense in $V$ and since both $\partial g^*|_X$ and $\partial h^*|_X$ are minimal $w^*$-cuscos on $V$. It follows that $\partial g^*=\partial h^*$ on $V$. Hence, there exists $r\in\R$ such that
$$g^*=h^*+r$$
on $V$. The above equality can be extended to $X$ since
$$\rm{int\; (dom}\;g^*|_X)=\rm{int\; (dom}\;h^*|_X).$$
Now we will prove this last equality. Taking conjugates in both sides of the inequality in (\ref{E6}) we have
$$g^*+g(x_0^*)\leq h^*.$$
Hence, $\rm{dom}$ ($h^*|_X$) $\subset$ $\rm{dom}$ $(g^*|_X)$ and so $\rm{int\; dom}$ $(h^*|_X)$ $\subset$ $\rm{int\; dom}$ $(g^*|_X)$. Therefore, we conclude the equality holds. It follows that
$$g^*=h^*+r$$
holds on $X$. Taking conjugates we obtain $g=h-r$ on $X^*$. Since $h(x_0^*)=0$, we conclude that $g(x_0^*)=-r$. Thus $g-g(x_0^*)=h$ and we complete the proof.

\end{proof}

\begin{thm}\label{T3} Let $g$ be a proper convex $w^*$-lower semicontinuous function
on $X^*$ such that $\rm{int\;( dom}$ $g)$ $\neq$ $\emptyset$ (resp. $\rm{int\; (dom}$ $(g^*|_X))$ $\neq$ $\emptyset$). Then given any point $x_0^*$ $\in$ $\rm{D}$ ($\partial g\cap X$)
and $x^*$ $\in$ $\rm{dom}$ $g$ (resp. $x^*\in X^*$), we have
$$g(x^*)=g(x_0^*)+\sup\{\sum_{i=0}^{n-1}\langle x_i,x_{i+1}^*-x_i^*\rangle +\langle x_n,x^*-x_n^*\rangle \},$$
where the above supremum is taken over all integers $n$, all $x_i^*\in X^*$ and all $x_i\in\partial g(x_i^*)\cap X$ for $i=0,1,\cdots,n$.
\end{thm}

\begin{proof}
 Let $T=\partial g\cap X$. By Theorem \ref{T1}, we obtain that D ($T$) is dense in $\rm{ int\; (dom}$ $g)$ (resp. $\rm{Im}$ $T$ is dense in $\rm{int\; (dom}$ $g^*|_X))$. Now the result follws from Theorem \ref{T2}, immediately.

\end{proof}

\begin{thm}Suppose that Banach space $X^*$ has the Radon-Nikodym property and the
convex function $g$ is $w^*$-epi-pointed. Let $x_0^*$ $\in$ $\rm{D}$ ($\partial g\cap X$) Then for every $x^*\in X^*$ we have
$$g(x^*)-g(x_0^*)=\sup\{\sum_{i=0}^{n-1}\langle x_i,x_{i+1}^*-x_i^*\rangle +\langle x_n,x^*-x_n^*\rangle\},$$
where the above supremum is taken over all integers $n$, all $x_1^*,x_2^*\cdots, x_n^*\in w^*-exp\;g$ and all $x_i\in\partial g(x_i^*)\cap X$ for $i=0,1,\cdots,n$.
\end{thm}

\begin{proof} Let $T: X^*\rightarrow 2^X$ be defined for all $x^*\in X^*$ by

$$T(x^*)=
\begin{cases}
\partial g(x^*)\cap X,& \text{$x^*\in\{x_0^*\}\cup w^*-\rm{exp}g,$}\\
\emptyset,& \text{$x^*\notin\{x_0^*\}\cup w^*-\rm{exp}g.$}
\end{cases}$$

Hence, Gr $T$ $\subset$ Gr $\partial g$ and $T$ is $w^*$-Cyclically monotone.
For a sequence $x_1^*,x_2^*\cdots, x_n^*\in$ D($T$), denote by $n_0$ the smaller index in $\{0,1,\cdots,n\}$ such that $x_i^*\neq x_0^*$ for all $i>n_0$. Therefore, $x_{n_0}^*=x_0^*$. It follows from the $w^*$-Cyclically monotonicity of $T$ that
$$\sum_{i=0}^{n_0-1}\langle x_i, x_{i+1}^*-x_i^*\rangle\leq0.$$
Hence, we may take the supremum among the sequence $x_{n_0+1}^*,x_{n_0+2}^*\cdots, x_n^*\in$ $w^*$-exp $g$ instead of the sequence $x_1^*,x_2^*\cdots, x_n^*$.
Since
$$w^*-\rm{Exp}\;g\subset\bigcup_{x^*\in w^*-\rm{exp}\; g}\partial g(x^*)\cap X\subset \rm{Im} T,$$
it follows from Proposition \ref{P5} that $\rm{Im}$ $T$ is dense in $\rm {int\; (dom}$ $g^*|_X)$ and the result follows from Theorem \ref{T2} or Theorem \ref{T3}.
\end{proof}

\section{Acknowledgements}

This work was partially done while the author was visiting Texas A$\&$M University and in Analysis and Probability Workshop at Texas A$\&$M University which was funded by NSF Grant. The author would like to thank Professor W.B. Johnson and Professor Th. Schlumprecht for the invitation.

\end{document}